\documentclass[12pt]{amsart}

\usepackage[english]{babel}
\usepackage{amssymb}
\usepackage{enumerate}

\newtheorem{thm}{Theorem}
\newtheorem{lem}{Lemma}
\newtheorem{rmk}{Remark}

\newcommand{\length}{\operatorname{length}}

\makeatletter
\@namedef{subjclassname@2010}{%
  \textup{2010} Mathematics Subject Classification}
\makeatother

\begin{document}

\title{On the Jayne-Rogers theorem }
\author{Sergey Medvedev}
\address{South Ural State University,  Chelyabinsk, Russia}
\email{medvedevsv@susu.ru}

\begin{abstract}
In 1982, J.E.~Jayne and C.A.~Rogers proved that a mapping $f \colon X \rightarrow Y$ of an absolute Souslin-$\mathcal{F}$ set $X$ to a metric space $Y$ is $\mathbf{\Delta}^0_2$-measurable if and only if it is piecewise continuous. We now give a similar result for a perfectly paracompact first-countable space $X$ and a regular space $Y$.
\end{abstract}



\maketitle

In 1982, J.E.~Jayne and C.A.~Rogers \cite{JR} proved the following

\begin{thm}[Jayne--Rogers]\label{JR}
If $X$ is an absolute Souslin-$\mathcal{F}$ set and $Y$ is a metric space, then $f \colon X \rightarrow Y$ is $\mathbf{\Delta}^0_2$-measurable if and only if it is piecewise continuous.
\end{thm}

The original proof of this theorem is long and quite complicated. Therefore Ka\v{c}ena, Motto Ros, and Semmes \cite{KMS} presented a new short proof of Theorem~\ref{JR}. Moreover, they showed that Theorem~\ref{JR} holds for a regular space $Y$. A further sharpening of the Jayne--Rogers theorem is connected with the study of $\mathbf{\Sigma}^0_2$-measurable mappings which are not piecewise continuous. Put this way, Solecki \cite[Theorem 3.1]{S98} demonstrated that, given a $\mathbf{\Sigma}^0_2$-measurable mapping $f \colon X \rightarrow Y$ of an analytic space $X$ to a separable metric space $Y$, either $f$ is piecewise continuous or $X$ contains a closed subset $D$ homeomorphic to the Cantor set $2^\omega$ such that the restriction $f | D$ is not $\mathbf{\Delta}^0_2$-measurable. By \cite[Corollary 6]{KMS}, the last statement is valid for  an absolute Souslin-$\mathcal{F}$ set $X$ and a regular space $Y$.

Piecewise continuous mappings between Polish spaces have recently been investigated by R.~Carroy and B.D.~Miller~\cite{CM}.

\textbf{Notation.} For all undefined terms see~\cite{Eng}.

A space $X$ is said to be a \textit{Souslin}-$\mathcal{F}$ set in $Z$ if $X$ is a result of the $\mathcal{A}$-operation applied to a system of closed subsets of $Z$. In particular, a space $X$ is an \textit{absolute Souslin}-$\mathcal{F}$ set if it is metrizable and a Souslin-$\mathcal{F}$ set in the completion $\tilde{X}$ of $X$ under its compactible metric.

A space $X$ is called \textit{perfectly paracompact} if $X$ is paracompact and each closed subset of $X$ is of type $G_\delta$ in $X$. A space $X$ is called a \textit{Baire space} if the intersection of countably many dense open sets in $X$ is dense. We call a space \textit{completely Baire} if every closed subspace of it is Baire.

A mapping $f \colon X \rightarrow Y$ is said to be

\begin{enumerate}[$\bullet$]

\item  $\mathbf{\Delta}^0_2$-\textit{measurable} if $f^{-1}(U) \in \mathbf{\Delta}^0_2(X)$ for every open set $U \subset Y$,

\item  $\mathbf{\Sigma}^0_2$-\textit{measurable} if $f^{-1}(U) \in \mathbf{\Sigma}^0_2(X)$ for every open set $U \subset Y$,

\item  \textit{piecewise continuous} if $X$ can be covered by a sequence $X_0, X_1,  \ldots$ of closed sets such that the restriction $f | X_n$ is continuous for every $n \in \omega$.
\end{enumerate}

Now we are ready to give the main result of the paper.

\begin{thm}\label{Th-m}
Let $Y$ be a regular space and $X$ be a Souslin-$\mathcal{F}$ set in a regular completely Baire space $Z$ such that $X$ is a perfectly paracompact space satisfying the first axiom of countability. Then $f \colon X \rightarrow Y$ is a $\mathbf{\Delta}_2^0$-measurable mapping if and only if it is piecewise continuous.
\end{thm}

\begin{proof}
To obtain a contradiction, we suppose that there is a $\mathbf{\Delta}_2^0$-measurable mapping $f  \colon X \rightarrow Y$ which is not piecewise continuous. Clearly, $f$ is $\mathbf{\Sigma}_2^0$-measurable. By Lemma~\ref{L-m}, there exist a countable subset $D = \{d_n \colon n \in \omega \}$ of $X$ and a sequence $\langle U_n \colon n \in \omega \rangle$ of disjoint open sets in $Y$ such that $\overline{D} \subset X$, every $f(d_n) \in U_n$, and $f(x) \notin \bigcup \{U_n \colon n \in \omega \}$ for every $x \in \overline{D} \setminus D$. This implies that
$ D = \overline{D} \cap f^{-1} \left( \bigcup \{U_n \colon n \in \omega \} \right)$.

$\overline{D}$ is a Baire space as a closed subset of the completely Baire space $Z$. Then $D$ is not a $G_\delta$-subset of $\overline{D}$; otherwise $\overline{D}$ would be a set of the first category. Since $D$ is not a $\mathbf{\Delta}_2^0$-subset of $X$, $f$ is not $\mathbf{\Delta}_2^0$-measurable, a contradiction.

For the other direction, consider a piecewise continuous mapping $f \colon X \rightarrow Y$. One can readily verify that $f^{-1}(A)$ is an $F_\sigma$-subset of $X$ for every closed $A \subset Y$. Then $f^{-1}(U)$ is $G_\delta$ in $X$ for every open $U \subset Y$. Moreover, $f^{-1}(U)$ is an $F_\sigma$-subset of $X$ for every open $U \subset Y$ because $X$ is a perfectly paracompact space.
Thus, $f$ is $\mathbf{\Delta}_2^0$-measurable.
\end{proof}

The basic Lemma~\ref{L-m} is given below.

\begin{rmk}
Theorem~\ref{Th-m} is proved in \textsc{ZFC}. Under the Martin Axiom, Banakh and Bokalo~\cite[Example 9.3]{BB} jointed with Zdomskyy constructed a $\mathbf{\Delta}_2^0$-measurable mapping $f \colon X \rightarrow Y$ between metrizable separable spaces which is not piecewise continuous.
\end{rmk}

\begin{rmk}
As noted above, when $X$ is an absolute Souslin-$\mathcal{F}$ set, the set $\overline{D}$ in the proof of Theorem~\ref{Th-m} can be chosen homeomorphic to the Cantor set $2^\omega$. This not to be true in general. For example, take the Bernstein set as the space $Z$. Recall that the Bernstein set is a metrizable completely Baire space which contains no closed subset homeomorphic to the Cantor set.
\end{rmk}

\begin{rmk}
Let us show that Theorem~\ref{Th-m} can be applied to a non-absolute Souslin-$\mathcal{F}$ set $X$. As an example, take a Souslin-$\mathcal{F}$ subset $X$ of the Sorgenfrey line $\mathbb{S}$.
\end{rmk}

To prove Theorem~\ref{Th-m}, we shall modify the technique due to Ka\v{c}ena, Motto Ros, and Semmes \cite{KMS}. Therefore, the terminology and methods from \cite{KMS} are applied.

The sets $A, B \subset Y$ are \textit{strongly disjoint} if $\overline{A} \cap \overline{B} = \emptyset$.

Given a mapping $f \colon X \rightarrow Y$, let us denote by $\mathcal{I}_f$ the family of all subsets $A \subset X$ for which there is a set $F \in \mathbf{\Sigma}^0_2(X)$ such that $A \subset F$ and the restriction $f | F$ is piecewise continuous. In particular, $f$ is piecewise continuous if and only if $X \in \mathcal{I}_f$. Clearly, the family $\mathcal{I}_f$ is a $\sigma$-ideal.

Let $x \in X$, $X^\prime \subset X$, and $A \subset Y$. Define $A^f = f^{-1}(Y \setminus \overline{A})$. The pair $(x, X^\prime)$ is said to be $f$-\textit{irreducible outside} $A$ if for every neighborhood $V \subset X$ of $x$ we have $A^f \cap X^\prime \cap V \notin \mathcal{I}_f$. Otherwise we say that $(x, X^\prime)$ is  $f$-\textit{reducible outside} $A$, i.e., there exist a neighborhood $V$ of $x$ and a set $F \in \mathbf{\Sigma}^0_2(X)$ such that $A^f \cap X^\prime \cap V \subset F$ and $f | F$ is piecewise continuous. Clearly, $x \in \overline{A^f \cap X^\prime}$ if $(x, X^\prime)$ is $f$-irreducible outside $A$.

\begin{lem}\label{L:if}
Let $f \colon X \rightarrow Y$ be a mapping from a perfectly paracompact space $X$ to a space $Y$. Given $X^\prime \subset X$, let $G = \bigcup\{U \mbox{ is open in } X \colon U \cap X^\prime \in \mathcal{I}_f \}$. Then $ X^\prime \cap G \in \mathcal{I}_f$.
\end{lem}

\begin{proof}
Denote by $\mathcal{U}$ the family $\{U \mbox{ is open in } X \colon U \cap X^\prime \in \mathcal{I}_f \}$. For every  $U \in \mathcal{U}$ choose a sequence of closed sets $F_n(U)$ such that $U \bigcap X^\prime \subset \bigcup \{F_n(U) \colon n \in \omega \}$ and all restrictions $ f | F_n(U)$ are continuous. Since $X$ is a perfectly paracompact space, we have $G = \bigcup\{G_k \colon k \in \omega \}$, where each $G_k$ is closed in $X$.  By the Michael theorem (see \cite[Theorem 5.1.28]{Eng}), $G$  is paracompact itself. Let  $\mathcal{V}$ be a locally finite open refinement of $\mathcal{U}$. For every $V \in \mathcal{V}$ fix $U_V \in \mathcal{U} $ such that $V \subset U_V$. Define $F^k_n = \bigcup\{\overline{V} \cap F_n(U_V) \cap G_k\colon V \in \mathcal{V}\}$. The local finiteness of the family $\mathcal{V}$ implies that all $F^k_n$ are closed in $X$ and all restrictions $ f | F^k_n$ are continuous. Since $ X^\prime \cap G \subset \bigcup\{F^k_n \colon k \in \omega, n \in \omega\}$, we obtain $ X^\prime \cap G \in \mathcal{I}_f$.
\end{proof}

The following lemma is a slight modification of \cite[Lemma 3]{KMS}; in \cite{KMS} it was proved for a metrizable space $X$.

\begin{lem}\label{L:k3}
Let $X$ be a perfectly paracompact space and $Y$ a regular space. Suppose $f \colon X \rightarrow Y$ is a $\mathbf{\Sigma}^0_2$-measurable mapping, $X^\prime$ is a subset of $X$, and $A \subset Y$ is an open set such that $X^\prime \subset A^f$. Then the following assertions are equivalent:

\begin{enumerate}[\upshape(i)]

\item $X^\prime \notin \mathcal{I}_f$,

\item there exist a point $x \in \overline{X^\prime} \cap A^f$ and an open set $U \subset Y$ strongly disjoint from $A$ such that $f(x) \in U$ and the pair $(x, X^\prime)$ is $f$-irreducible outside $U$.
\end{enumerate}
\end{lem}

\begin{proof}
(ii) $\Rightarrow$ (i): If $(x, X^\prime)$ is $f$-irreducible outside $U$, then $U^f \cap X^\prime \cap X \notin \mathcal{I}_f$. Hence, $X^\prime \notin \mathcal{I}_f$.

(i) $\Rightarrow$ (ii):
Denote by $G$ the open set $\bigcup\{O \mbox{ is open in } X \colon O \cap X^\prime \in \mathcal{I}_f \}$. By Lemma~\ref{L:if} we have $ X^\prime \cap G \in \mathcal{I}_f$.

Assume toward a contradiction that (ii) does not hold, i.e., for every $x \in \overline{X^\prime} \cap A^f$ and every open set $U \subset Y$ strongly disjoint from $A$ such that $f(x) \in U$ we have that $(x, X^\prime)$ is $f$-reducible outside $U$.

The intersection $Z = A^f \cap (\overline{X^\prime} \setminus G)$ is an $F_\sigma$-set in $X$. We claim that the restriction $f | Z$ is continuous. Suppose otherwise, so that there are $x \in Z$ and an open set $U \subset Y$ such that $f(x) \in U$ and there is no neighborhood $V$ of $x$ with $f(V \cap Z) \subset \overline{U}$. Because $f(x) \notin \overline{A}$ by the assumption, we can assume that $\overline{U} \cap \overline{A} = \emptyset$. Fix a neighborhood $V$ of $x$ satisfying $U^f \cap X^\prime \cap V \in \mathcal{I}_f$. By our hypothesis there is $x^\prime \in V \cap Z$ such that $f(x^\prime) \notin \overline{U} \cup \overline{A}$. Using regularity of $Y$, we can find a neighborhood $U^\prime$ of $f(x^\prime)$ which is strongly disjoint from $U$ and $A$. Let now $V^\prime$ be a neighborhood of $x^\prime$ given by the failure of (ii), i.e. $(U^\prime)^f \cap X^\prime \cap V^\prime \in \mathcal{I}_f$. Since $\overline{U} \cap \overline{U^\prime} = \emptyset$, we have
$ X^\prime \cap V \cap V^\prime \subset U^f \cup (U^\prime)^f$. Therefore $ X^\prime \cap V \cap V^\prime \in  \mathcal{I}_f$. In other words, $ X^\prime \cap V \cap V^\prime \subset X^\prime \cap G$. But this implies $x^\prime \notin Z$, a contradiction. Thus $f | Z$ is continuous. Hence $X^\prime \cap Z \in \mathcal{I}_f$.

On the whole, we get $X^\prime = (X^\prime \cap Z) \cup (X^\prime \cap G) \in \mathcal{I}_f$, a contradiction with (i).
\end{proof}

\begin{lem}[{\cite{KMS}}]\label{L:k4}
Let $f \colon X \rightarrow Y$ be a mapping of a space $X$ to a space $Y$, $x \in X$, $X^\prime \subset X$, $A \subseteq Y$, and let $U_0, \ldots, U_k$ be a sequence of pairwise strongly disjoint open subsets of $Y$. If $(x; X^\prime)$ is $f$-irreducible outside $A$, then there is at most one $i \in \{ 0, \ldots, k \}$ such that $(x, X^\prime)$ is $f$-reducible outside $A \cup U_i$.
\end{lem}

For the sake of completeness, we reproduce the proof of Lemma~\ref{L:k4} from~\cite{KMS}.

\begin{proof}
Assume that there are two indices $i, j \in \{0, \ldots , n \}$, $i \neq j$, such that $(x, X^\prime)$ is $f$-reducible outside both $A \cup U_i$ and $A \cup U_j$. Then there are neighborhoods $V_i$ and $V_j$ of $x$ such that $(A \cup U_i)^f \cap X^\prime \cap V_i \in \mathcal{I}_f$ and $(A \cup U_j)^f \cap X^\prime \cap V_j \in \mathcal{I}_f$. Since $U_i$
and $U_j$ are strongly disjoint, this implies that
\[A^f \cap X^\prime \cap V_i \cap V_j  \in \mathcal{I}_f,\]
and thus $V_i \cap V_j$ contradicts the fact that  $(x, X^\prime)$ is $f$-irreducible outside $A$.
\end{proof}

\begin{lem}\label{L-m}
Let $Y$ be a regular space and $X$ be a Souslin-$\mathcal{F}$ set in a regular space $Z$ such that $X$ is a perfectly paracompact space satisfying the first axiom of countability. Let $f \colon X \rightarrow Y$ be a $\mathbf{\Sigma}_2^0$-measurable mapping which is not piecewise continuous.

Then there exist a countable subset $D = \{d_n \colon n \in \omega \}$ of $X$ and a sequence $\langle U_n \colon n \in \omega \rangle$ of disjoint open sets in $Y$ such that
\begin{enumerate}[\upshape (i)]
\item $f(d_n) \in U_n$ for every $n \in \omega $,

\item $D$ is homeomorphic to the space of rational numbers,

\item $\overline{D} \subset X$, where the bar denotes the closure in $Z$,

\item $f(x) \notin \bigcup \{U_n \colon n \in \omega \}$ for every $x \in \overline{D} \setminus D$.
\end{enumerate}
\end{lem}

\begin{proof}
Let $X$ have a representation $X = \bigcup \left\{\bigcap \{F_{\alpha|n} \colon n \in \omega \} \colon \alpha \in \omega^\omega \right\}$, where each $F_{\alpha|n}$ is closed in $Z$ and $F_{\alpha|n+1} \subset F_{\alpha|n}$. For every $t \in \omega^{<\omega}$ we denote
\[ X_t= \bigcup \left\{\bigcap\nolimits_{n \in \omega} F_{t \hat{\:} \beta|n} \colon \beta \in \omega^\omega \right\}.
\]
Clearly,  $X_t \subset F_t$ and $X_t = \bigcup\{X_{t \hat{\:} n} \colon n \in \omega \}$ for every $t \in \omega^{<\omega}$.

Denote by $2^{<\omega}$ the set of all binary sequences of finite length. The construction will be carried out by induction with respect to the order $\preceq$ on $2^{<\omega}$ defined by
\[s \preceq t \, \Longleftrightarrow \, \length(s) < \length(t) \vee (\length(s) = \length(t) \wedge s \leq_\mathrm{lex} t), \]
where $\leq_\mathrm{lex}$ is the usual lexicographical order on $2^{\length(s)}$. We write $s \prec t$ if $s \preceq t$ and $s \neq t$.

The map $\upsilon$ assigns to each $t \in 2^{<\omega}$ the length of a string of zeros at the end of $t$; for example, $\upsilon(110100)=2$.

We will construct a sequence $\langle x_s \colon s \in 2^{<\omega} \rangle$ of points of $X$, a sequence $\langle W^s_n \colon  n \in \omega  \rangle$ of open subsets of $Z$, a sequence $\langle V_s \colon s \in 2^{<\omega}\rangle$ of subsets of $X$, a sequence $\langle U_s \colon s \in 2^{<\omega} \rangle$ of open subsets of $Y$, and a mapping $h \colon 2^{<\omega} \rightarrow \omega^{<\omega}$ such that for every $s \in 2^{<\omega}$:

\begin{enumerate}[\upshape (1)]
\item $\langle X \cap W^s_n \colon  n \in \omega  \rangle$ forms a base at the point $x_s$ with respect to $X$ and $\overline{W^s_{n+1}} \subset W^s_n$ for each $n \in \omega $,

\item if $t \subset s$ then $V_s \subset V_t$,

\item $x_s \in X \cap \overline{V_s}$ and $\overline{V_s} \subset W^s_0$,

\item if $s = t \hat{\:} 0$ then $x_s = x_t$ and $U_s = U_t$,

\item $f(x_s) \in U_s$,

\item if the last digit of $s$ is 1 then $f(X \bigcap \overline{V_s}) \cap \overline{\bigcup_{r \prec s} U_r}= \emptyset$,

\item $(x_t, V_t)$ is $f$-irreducible outside $A$ for every $t \preceq s$, where $A = \bigcup \{U_r \colon r \preceq s \}$,

\item the family $\{V_t \colon t \in 2^{n} \}$ is pairwise strongly disjoint in $X$ for every $n \in \omega$,

\item the family $\{U_t \colon t \preceq s \; \wedge$ (the last digit of $t$ is 1) or $t = \emptyset \}$ is pairwise strongly disjoint,

\item if $t \varsubsetneq s$ then $h(t) \subset h(s)$ and if, moreover, the last digit of $s$ is 1 then $\length(h(t)) < \length(h(s))$,

\item $V_s \subset X_{h(s)} \subset F_{h(s)}$,

\item $\overline{V_{s \hat{\:} 0}} \subset W^s_{1+\upsilon(s)}$ and $\overline{V_{s \hat{\:} 1}} \subset W^s_{\upsilon(s)} \setminus \overline{W^s_{1+\upsilon(s)}}$.
\end{enumerate}

For the base step of the induction, let $x$ and $U$ be given as in Lemma~\ref{L:k3} applied to $X^\prime = X$ and $A = \emptyset$. Then put $ x_\emptyset = x$, $U_\emptyset = U$, and $h(\emptyset) = \emptyset$. Choose a sequence $\langle W^\emptyset_n \colon  n \in \omega  \rangle$ satisfying condition (1) for $s = \emptyset$. Define $V_\emptyset = X \cap W^\emptyset_1$.

Assume that $x_t$, $W^t_n$, $V_t$, and $U_t$ have been constructed for every $t \prec s \hat{\:} 0$. Let $A = \bigcup \{U_t \colon t \prec s \hat{\:} 0 \}$ and $O = Y \setminus \overline{A}$. By the inductive hypothesis, condition (7) says that $(x_s, V_s)$ is $f$-irreducible outside $A$. Then
\[f^{-1}(O) \cap V_s \cap W^s_{\upsilon(s)}= A^f \cap V_s \cap W^s_{\upsilon(s)} \notin \mathcal{I}_f.\]
The decreasing sequence $\langle X \cap W^s_n \colon  n \in \omega  \rangle$ forms a base at $x_s$. Since $X$ is perfectly paracompact, every set $X \cap (W^s_n \setminus \overline{W^s_{n+1}})$ is $F_\sigma$ in $X$. Now we can find some $k > 0$ with
\[f^{-1}(O) \cap V_s \cap \left(W^s_{\upsilon(s)} \setminus \overline{W^s_{k+\upsilon(s)}} \right) \notin \mathcal{I}_f.\]
To simplify notation, we can assume that $k=1$. Since $f$ is $\mathbf{\Sigma}_2^0$-measurable, there exists a set $C$ closed in $X$ such that
\[C \subset f^{-1}(O) \cap \left(W^s_{\upsilon(s)} \setminus \overline{W^s_{1+\upsilon(s)}} \right)\]
and $C \cap V_s \notin  \mathcal{I}_f$. By construction, $V_s \subset X_{h(s)} = \bigcup \{X_{h(s)\hat{\:} n} \colon n \in \omega \}$. Then for some $n$ we have
\[X^\prime = C \cap V_s \cap X_{h(s)\hat{\:} n} \notin  \mathcal{I}_f.\]

Put $h(s \hat{\:} 0) = h(s)$ and $h(s \hat{\:} 1) = h(s) \hat{\:} n$.

Next, set $x_{s \hat{\:} 0} = x_s$, $U_{s \hat{\:} 0} = U_s$,  $W^{s \hat{\:} 0}_n = W^s_{n+1}$, and $V_{s \hat{\:} 0} = V_s \cap W^{s \hat{\:} 0}_1$.

\textbf{Claim.} There are $x_{s \hat{\:} 1} \in \overline{X^\prime} $ and $U_{s \hat{\:} 1} \subset Y$ such that $f(x_{s \hat{\:} 1}) \in U_{s \hat{\:} 1}$, $U_{s \hat{\:} 1}$ is open and strongly disjoint from $A$,
$(x_t, V_t)$ is $f$-irreducible outside $A \cup U_{s \hat{\:} 1}$ for every $t \prec s \hat{\:} 1$ and $(x_{s \hat{\:} 1}, X^\prime)$ is $f$-irreducible outside $A \cup U_{s \hat{\:} 1}$.

\textsc{Proof of the Claim.} Let $k = | \{t \in 2^{<\omega} \colon t \prec s \hat{\:} 1 \} |$. Using Lemma~\ref{L:k3}, for $\jmath = 0, \ldots, k$ recursively construct $x_\jmath$ and $U_\jmath$ such that $f(x_\jmath) \in U_\jmath$, $U_\jmath$ is open and strongly disjoint from $A \cup U_{<\jmath}$ (where $U_{<\jmath} = \emptyset$ if $\jmath = 0$ and $U_{<\jmath} = \bigcup \{U_i \colon i < \jmath \}$ otherwise), and $(x_\jmath, X^\prime \cap (U_{<\jmath})^f)$ is $f$-irreducible outside $U_\jmath$. According to Lemma~\ref{L:k4}, for each $t \prec s \hat{\:} 1$ there is at most one $j_t \in \{0, \ldots, k \}$ such that $(x_t, V_t)$ is $f$-reducible outside $A \cup U_{j_t}$. The pigeonhole principle implies that the claim is satisfied with $x_{s \hat{\:} 1} = x_{\jmath^*}$ and $U_{s \hat{\:} 1} = U_{\jmath^*}$ for some $\jmath^* \in \{0, \ldots, k \}$. $\triangle$

Choose a sequence $\langle W^{s \hat{\:} 1}_n \colon  n \in \omega  \rangle$ of open subsets of $Z$ such that $\langle X \cap W^{s \hat{\:} 1}_n \colon  n \in \omega  \rangle$ is a base at the point $x_{s \hat{\:} 1}$ satisfying $\overline{W^{s \hat{\:} 1}_0} \subset W^s_{\upsilon(s)} \setminus \overline{W^s_{1+\upsilon(s)}}$ and $\overline{W^{s \hat{\:} 1}_{n+1}} \subset W^{s \hat{\:} 1}_n$. Finally, set $V_{s \hat{\:} 1} = X^\prime \cap W^{s \hat{\:} 1}_1$. Obviously, $V_{s \hat{\:} 1} \subset V_s$.

By construction, $\overline{V_{s \hat{\:} 0}} \subset W^s_{1+\upsilon(s)}$ and $\overline{V_{s \hat{\:} 1}} \subset W^s_{\upsilon(s)} \setminus \overline{W^s_{1+\upsilon(s)}}$. Then $\overline{V_{s \hat{\:} 0}} \cap \overline{V_{s \hat{\:} 1}} = \emptyset$. Together with (2), this implies condition (8). One can check that all the conditions (1)--(12) are satisfied.

Define the set $D = \{x_s \colon s \in 2^{<\omega} \}$. The set $f(D)$ is relatively discrete in $Y$ because $f(x_s) \in U_s$ and $U_s \cap U_t = \emptyset$ whenever $x_s \neq x_t$. Clearly, the restriction $f | D$ is a bijection.

The set $D$ is countable and has a countable base at each point; hence, the space $D$ is second-countable. By the Urysohn theorem \cite[Theorem 4.2.9]{Eng}, $D$ is metrizable. From conditions (1) and (3) it follows that $D$ has no isolated points. The Sierpi\'{n}ski theorem \cite{Sie} implies that  $D$ is homeomorphic to the space of rationals.

Let us check that $\overline{D} \subset X$. From conditions (2)--(4) it follows that $D \subset V^*_n = \bigcup \{\overline{V_s} \colon s \in 2^n \}$ for each $n \in \omega$. Since $2^n$ is a finite set, $V^*_n$ is closed in $Z$. Then $\overline{D}$ belongs to a closed subset $V^* = \bigcap \{V^*_n \colon n \in \omega \}$ of $Z$. Clearly, $D \subset X$. Given a point $x \in \overline{D} \setminus D$, fix $s(x,n)$ such that $x \in \overline{V_{s(x,n)}}$ and $s(x,n) \in 2^n$ for each $n$. Using conditions (2) and (8), find a unique $\alpha_x \in 2^\omega$ with $s(x,n) = \alpha_x | n$. Since $x \notin D$, the sequence $\alpha_x$ contains infinitely many digits $1$. By (10), the sequence $\langle h(\alpha_x | n) \colon n \in \omega \rangle$ has infinitely many distinct members. From (11) it follows that $x \in \bigcap \{\overline{V_{\alpha_x | n}} \colon n \in \omega \} \subset \bigcap \{F_{h(\alpha_x | n)} \colon n \in \omega \} \subset X$. Therefore, $\overline{D} \subset X$.

To obtain (iv), take a point  $x \in \overline{D} \setminus D$. Striving for a contradiction, suppose that $f(x) \in U_r$ for some $r \in 2^{<\omega}$. By the above, $x \in \bigcap \{\overline{V_{\alpha_x | n}} \colon n \in \omega \}$. Since the sequence $\alpha_x$ contains infinitely many $1$, there is $n$ such that the last digit of $\alpha_x | n$ is $1$ and $r \prec \alpha_x | n$. From condition (6) it follows that $f(x) \notin \overline{U_r}$, a contradiction.

It remains to number the distinct elements of $D$ as $\{d_n \colon n \in \omega \}$.
\end{proof}

\end{document}